\documentclass[12pt,oneside,reqno]{amsart}
\newtheorem{theorem}{Theorem}[section]
\newtheorem{lemma}[theorem]{Lemma}

\newtheorem{proposition}[theorem]{Proposition}
\newtheorem{definition}[theorem]{Definition}

\title[Analysis of First-Order Linear Pursuit and Evasion Differential Games with Grönwall-Type Constraints]
   {Analysis of First-Order Linear Pursuit and Evasion Differential Games with Grönwall-Type Constraints}
  \date{}

 \author[David Terna Gbande, Abbas Ja'afaru Badakaya, Jamilu Adamu and Mehdi Salimi ]{David Terna Gbande$^1$, Abbas Ja'afaru Badakaya$^2$, Jamilu Adamu$^3$, Mehdi Salimi $^{4*}$(Corresponding author)\\\\
$^1$Department of Mathematical Sciences, Bayero University, Kano, and Federal Polytechnic Wannune, Benue State, Nigeria\\
$^2$Department of Mathematical Sciences, Bayero University, Kano, Nigeria\\
$^3$Department of Mathematics, Federal University, Gashu'a, Nigeria\\
$^4$Mathematics Department, Kwantlen Polytechnic University, BC, Canada}
\thanks{princedavison4@gmail.com (David Terna Gbande); ajbadakaya.mth@buk.edu.ng (Abbas Ja'afaru Badakaya); jamiluadamu88@gmail.com (Jamilu Adamu); mehdi.salimi@kpu.ca (${}^*$Corresponding author)}
\begin{document}
\maketitle
\begin{abstract}
 We examine pursuit and evasion problems in the space $ \mathbb{R}^{n} $ involving a lone pursuer and evader. Motion of each player is governed by a first-order linear differential equation. The players' control functions are subject to the Gronwall-type inequality. The game's duration is fixed and represented by a positive number $\vartheta $. When the state of a pursuer coincides with that of the evader, we then say the pursuit is completed. On the contrary, evasion is possible when throughout the game there is no agreement between the states of pursuer and evader. The result obtained related to the pursuit problem depends on the attainability domains of the two players. On the other hand, one sufficient condition is given for evasion to be possible.
\end{abstract}
\maketitle

\section{introduction}
Analysis of First-Order Linear Pursuit and Evasion Differential Games with Grönwall-Type Constraints

\noindent The idea of differential game problems, which encompass pursuit and evasion games is of interest to many researchers. In many of the investigations works, player motion  are piloted by  
\begin{equation}\label{AT3}
\left\{ \begin{array}{ll}
P: \dot{x}(t) +\alpha x(t)  = c(t)u(t), \ \ \ x(0) = x_{0},  \\
E:  \dot{y}(t) +\alpha y(t) = \  c(t)v(t), \ \ \ y(0) = y_{0}.
\end{array}\right
.
\end{equation}
The pursuer and evader's control functions are denoted by $u(\cdot)$ and $ v(\cdot)$ respectively,  whereas $c(t)$ represents a specified or arbitrary scalar function and $\alpha $ is a real constant.\\

\noindent The differential game, in which the players’ motions are governed by (\ref{AT3}) with $\alpha = 0$ and $c(t) = 1$, subject to an integral constraint on the players’ controls, was recently investigated in \cite{REFBad} and \cite{REFM2}.\\
In contrast, \cite{REFBad} investigated cases where the players’ control functions are constrained geometrically, while \cite{REFM2} examined a problem in which the pursuer’s control function is subject to an integral constraint and the evader’s control function is subject to a geometric constraint.\\

\noindent Pursuit and evasion problem is studied by Rilwan et al. in \cite{TY21} where players' control functions is subject to Gronwall-type constraints, the dynamic motion of players also obeyed equation  (\ref{AT3}) for $\alpha = 0$. \\

\noindent When $\alpha$ is a nonzero constant and $c(t)=1$, Ibragimov in \cite{TY5} studied a game in which the players’ motions are governed by (\ref{AT3}) under integral constraints on their controls in the Hilbert space $\mathbb{R}^{n}$. Conditions ensuring the completion of pursuit were established.\\[3mm] 
\noindent Salimi and Ferrara \cite{TY22} studied a differential game where a single evader is pursued by a finite or countable set of pursuers. Each player’s control functions are governed by integral constraints, and the game is played over a fixed time horizon. The payoff is defined as the distance between the evader and the closest pursuer at the terminal moment. In their analysis, the authors introduced the notion of the game’s value and determined the pursuers’ optimal strategies. A notable feature of their work is that the pursuers’ energy resources are treated as independent from those of the evader in carrying out the pursuit.\\

\noindent Optimal pursuit time is found by Ibragiov at al. in \cite{TY23}. The game problem was study in the space $ \mathbb{R}^{n} $ with player control functions subject to integral constraints.\\
\noindent Furthermore, a game problem where dynamic motions of players described by (\ref{AT3}) where $\alpha$ is not a zero constant and $c(t)=1$ is also studied by Usman at al. in \cite{TY26}. Control functions of the players satisfied both geometric and integral constraints. They found conditions for pursuit to be finished in each instance.\\[3mm]
\noindent The concept of how Gronwall-type constraints can be used in the theory of differential game when dynamic motions of players described by first order linear equation (\ref{AT3}) where $\alpha$ is not a zero constant and $c(t)=1$  was first considered by Samatov et al. in \cite{TY25}. They solved pursuit problem with Gronwall-type constraints imposed on control parameter of players. In \cite{Sema}, a differential evasion game in a Hilbert space involving a finite number of pursuers and a single evader is studied. The players’ control functions are subject to geometric constraints. The authors resolve the game by constructing a strategy for the evader that ensures successful evasion.\\

\noindent In the work of Rilwan and co-authors  \cite{TY21} that studied a fixed duration pursuit and evasion differential game involving one pursuer and one evader, with Gronwall-type constraints imposed on the control functions of all players and dynamics of the players are governed by the dynamic equation (\ref{AT3}) for $\alpha = 0$. These result has systematically advanced the theory of pursuit and evasion differential games under Gronwall-type constraints, addressed the existence of permissible strategies and the derivation of guaranteed capture and escape conditions. Despite the substantial progress achieved in the literature, several challenges remain unresolved. In particular, there is a need for a unified treatment of pursuit and evasion differential games governed by first-order linear dynamics (\ref{AT3}), for $\alpha \neq 0$ under Gronwall-type constraints. The present work is motivated by these gaps and aims to contribute to the theory by extending existing results, refining analytical techniques, and providing new results on guaranteed pursuit and evasion  differential games with Gronwall-type control constraints.

\section{Formulation of the problem}
\noindent This Research work extends the theory of constrained differential games beyond the classical framework of simple motion player dynamics, motivated by the pursuit and evasion frameworks of Rilwan et al. \cite{TY21} which studied pursuit and evasion games under Gronwall-type control constraints using simple linear dynamics, this research addresses the problem of analyzing pursuit and evasion differential games involving one pursuer and a alone evader governed by modified first-order linear differential equations, with the aim of establishing conditions for capture and escape under Gronwall-type constrained controls.\\
\noindent In the space $\mathbb{R}^n$, we examine a differential game problem where the players' dynamics is described by (\ref{AT3}), in which $x, y,x_{0},y_{0}, u, v \in\mathbb{R}^n $, $ c(t)$ is an arbitrary non-negative scalar function, and $ \alpha $ is a non-zero constant.

\begin{definition}
 A measurable functions $u(t) = (u_{1} ,u_{2}, \dots, u_{n})  $  and $v(t) = (v_{1} ,v_{2}, \dots, v_{n})  $ that satisfy the inequalities
 \begin{equation}\label{DAV1}
|u(t)|^{2} \leq \rho^{2} + 2k \int_{0}^{t}|u(s)|^{2}ds  \ \ \ t \geq 0,
\end{equation}
\begin{equation}\label{DAV2}
 |v(t)|^{2} \leq \sigma^{2} + 2k \int_{0}^{t}|v(s)|^{2}dt  \ \ \ t \geq 0,
\end{equation} 
where $\rho, \sigma $ and $k$ are given positive values, are called permissible control of the  pursuer $ P $ and the evader $ E $ in that order with respect to Gronwall-type constraint.
\end{definition}
 
\noindent We let $ \bar{U} $ and $ \bar{V} $ denote the set of every permissible controls of the pursuer and evader.\\
Given $ u(\cdot) \in \bar{U}$ and $ v(\cdot)\in \bar{V},$ and for any starting positions 
$ x_{0}, y_{0} $, the trajectory of the pursuer $x(t)$ and the evader $y(t)$ at any time $ t > 0 $ are given by the following equations 
\begin{equation}\label{DAV3}
x(t) = x_{0}e^{-\alpha t} +  \int_{0}^{t}e^{-\alpha(t-s)}c(s)u(s)ds ,
\end{equation}
\begin{equation}\label{DAV4}
y(t) = y_{0}e^{-\alpha t} +  \int_{0}^{t}e^{-\alpha(t-s)}c(s)v(s)ds.
\end{equation}

\begin{definition}
 A function $U(t)=U(x_{0}, y_{0}, t, v(t)),$ is called the pursuer’s strategy, if for any $ v(\cdot)\in \bar{V} $, the system 
  \begin{equation}\label{AMIMM}
\left\{ \begin{array}{ll}
P: \dot{x}(t) + \alpha x(t)= c(t)U(x_{0}, y_{0}, t, v(t)), \ \ \ x(0) = x_{0}, \\
E:  \dot{y}(t)+\alpha y(t) = \ c(t)v(t), \ \ \ y(0) = y_{0},
\end{array}\right.
\end{equation}
is uniquely solvable with solution $(x(t), y(t))$ and 

 $$ |U(x_{0}, y_{0}, t, v(t))|^{2} \leq \rho^{2} + 2k \int_{0}^{t}|U(x_{0}, y_{0}, t, v(t))|^{2}ds. $$
If every control that the strategy $ U $ generates is permissible, then the strategy is considered permissible.
\end{definition}

\begin{definition}
 A function $V(x_{0}, y_{0}, t),$ is called the evader’s strategy, if for any $ u(\cdot)\in \bar{U} $, the system 
  \begin{equation}\label{MMM}
\left\{ \begin{array}{ll}
P: \dot{x}(t) +\alpha x(t)= c(t)u(t), \ \ \ x(0) = x_{0} ,\\
E:  \dot{y}(t)+\alpha y(t) = \  c(t)V(x_{0}, y_{0}, t), \ \ \ y(0) = y_{0},  
\end{array}\right.
\end{equation} is uniquely solvable with solution $(x(t), y(t))$ and,
$$ |V(x_{0}, y_{0}, t)|^{2} \leq \rho^{2} + 2k \int_{0}^{t}|V(x_{0}, y_{0}, t)|^{2}ds. $$ 
If every control that the strategy $ U $ generates is permissible, then the strategy is considered permissible.
\end{definition} 
\begin{definition}
If the pursuer’s strategy guarantees that, for every permissible control of the evader, the equality $x(\tau) = y(\tau)$ holds for some $\tau \in (0,\theta]$, then the pursuit is said to be completed.
\end{definition}
\begin{definition}
Evasion is said to be possible if there exits a strategy of the evader such
that for any permissible control of the pursuer the relations $x(t) \neq y(t)$
hold for all  $t > 0.$
\end{definition}
\textbf{Questions}:\\
The followings are the research questions:
\begin{enumerate}
\item[i.] Which sufficient condition(s) guarantee the completion of pursuit within a finite time?
\item[ii.] Which sufficient condition(s) ensure that evasion remains possible at all times?
\end{enumerate}

\section{Main results}
\noindent The primary outcome of the paper is presented in this section; earlier, some previous findings that are helpful in demonstrating the research's main outcomes are presented. 
\begin{lemma}\label{TB} \cite{TY21}
	Let $ \varphi $ and $ k $ be positive real values, and  $ w(t)$ be a measurable function for $ t \geq 0$, then, the inequality 
\end{lemma}
\begin{equation}\label{TB1}
 \displaystyle|w(t)|\leq \varphi e^{k\int_{0}^{t}c(s)ds}, 
\end{equation}
is true when
\begin{equation}\label{TB2}
|w(t)|^2\leq \varphi^2 + 2k\int_{0}^{t}c(s)|w(s)|^2ds, 
\end{equation} 
holds.
\begin{lemma}\label{TK} \cite{TY21}
The relation\\
\begin{equation}\label{DAV10a}
 e^{2k \int_{0}^{t} c(s)ds} = 1 + 2k \int_{0}^{t} c(s)e^{2k \int_{0}^{s} c(r)dr}ds,
\end{equation}
holds for all non-negative real-value functions $ c(s) $. 
\end{lemma}

\noindent \textbf{Player's attainability domain}

\begin{proposition}
The attainability domain of
\begin{itemize}
\item[i.] the pursuer $P$ from the initial state $x_{0}$ at time $t=0$
to the time $t=\vartheta $ is the ball $H_{P}(x_{0}e^{-\alpha\vartheta}, R(0,\vartheta) )$  
of radius $R(0,\vartheta) $  centred at $x_{0}e^{-\alpha\vartheta}$,\\
where $\displaystyle R(0,\vartheta)  = \left( \rho^2 + 2k\int_{0}^{\vartheta} c(s)|u(s)|^2 ds \right)^\frac{1}{2} \int_{0}^{\vartheta} e^{-\alpha(\vartheta -s)}c(s)ds$.

\item[ii.] the evader $E$ from the initial state $y_{0}$ at time $t=0$
to the time $t=\vartheta $ is the ball $H_{E}(x_{0}e^{-\alpha\vartheta}, r(0,\vartheta) )$  
of radius $r(0,\vartheta) $  centred at $y_{0}e^{-\alpha\vartheta}$,\\
where $\displaystyle r(0,\vartheta)  = \left( \sigma^2 + 2k\int_{0}^{\vartheta} c(s)|v(s)|^2 ds \right)^ \frac{1}{2}\int_{0}^{\vartheta} e^{-\alpha(\vartheta -s)}c(s)ds $.
 \end{itemize}  
\end{proposition}
 \begin{proof}
 For the verification of (i), with (\ref{DAV1}) and (\ref{DAV3}) we obtain
 \begin{align*}
\left\vert x(\vartheta )- x_{0}e^{-\alpha \vartheta } \right\vert 
 &\leq   \int_{0}^{\vartheta}e^{-\alpha(\vartheta-s)} c(s)|u(s)|ds
\\&\leq   \int_{0}^{\vartheta}e^{-\alpha(\vartheta-s)} c(s)\left( \rho^2 + 2k\int_{0}^{s} c(r)|u(r)|^{2}dr \right)^\frac{1}{2}ds
\\&\leq \left( \rho^2 + 2k\int_{0}^{\vartheta} c(r)|u(r)|^2 dr \right) ^\frac{1}{2}\int_{0}^{\vartheta} e^{-\alpha(\vartheta -s)}c(r)dr
\\&\leq \left( \rho^2 + 2k\int_{0}^{\vartheta} c(s)|u(s)|^2 ds \right) ^\frac{1}{2}\int_{0}^{\vartheta} e^{-\alpha(\vartheta -s)}c(r)dr:=\displaystyle R(0,\vartheta).
\end{align*}
 Hence $ \left\vert x(\vartheta )- x_{0}e^{-\alpha \vartheta } \right\vert \leq  R(0,\vartheta)$.\\
 
\noindent Also let $x^{*}\in H_{P}(x_{0}e^{-\alpha\vartheta}, R(0,\vartheta) )$, if the pursuer uses the control
  $$   u(t)=\frac{(x^{*}-x_{0}e^{-\alpha\vartheta})}{\int_{0}^{\vartheta} e^{-\alpha(\vartheta -t)}c(t)dt}, \ \ 0\leq t \leq \vartheta, $$
  then we have
\begin{align*}
 x(\vartheta)&= x_{0}e^{-\alpha\vartheta}+ \int_{0}^{\vartheta}e^{-\alpha(\vartheta-t)} c(t)u(t)dt
\\&= x_{0}e^{-\alpha\vartheta}+ \int_{0}^{\vartheta}e^{-\alpha(\vartheta-t)} c(t)\frac{(x^{*}-x_{0}e^{-\alpha\vartheta})}{\int_{0}^{\vartheta} e^{-\alpha(\vartheta -t)}c(t)dt}dt
=x^{*}.
\end{align*}
This proves (i). Using similar argument, one can prove (ii). With this, the verification of the proposition is complete.
\end{proof}

\noindent \textbf{Pursuit Differential Game} \\
To state the criteria for pursuit to be completed in the game problem (\ref{AT3}-\ref{DAV2}), we introduce the following notations. Let $\varepsilon = \frac{y_{0} - x_{0}}{|y_{0} - x_{0}|}$ and $ \lambda = \rho^2 -\sigma^2 $. Define the half space $ X $ by
\begin{equation}\label{DAV11}
 X: =\lbrace z \in\mathbb{R}^n :2\langle y_{0}e^{-\alpha\vartheta} - x_{0}e^{-\alpha\vartheta} , z \rangle \leq \Delta + |y_{0}e^{-\alpha\vartheta}|^2 -|x_{0}e^{-\alpha\vartheta}|^2 \rbrace, 
\end{equation}
where
\begin{align*}
\Delta &=R^2(0,\vartheta) -r^2(0,\vartheta)
\\&= \left( \rho^2 + 2k\int_{0}^{\vartheta} c(s)|u(s)|^2 ds \right)  \left(\int_{0}^{\vartheta} e^{-\alpha(\vartheta -s)}c(s)ds\right)^2 
\\&- \left( \sigma^2 + 2k\int_{0}^{\vartheta} c(s)|v(s)|^2 ds \right)  \left(\int_{0}^{\vartheta} e^{-\alpha(\vartheta -s)}c(s)ds\right)^2 
\\&=\lambda e^{2k \int_{0}^{\vartheta} c(s)ds}\left(\int_{0}^{\vartheta} e^{-\alpha(\vartheta -s)}c(s)ds\right)^2.
\end{align*}
 
\begin{theorem}\label{TTT}
In game (\ref{AT3}) - (\ref{DAV2}), pursuit can be completed if $y(\vartheta) \in X$ and $\lambda \geq 0$.

\end{theorem}
\begin{proof}
Allow the pursuer to employ this strategy:
\begin{equation}\label{DAV7} 
U(t, v) =  v(t)-\langle v(t), \varepsilon \rangle \varepsilon + \varepsilon \sqrt{\lambda e^{2k \int_{0}^{t} c(s)ds}+\langle v(t), \varepsilon\rangle^{2}},
\end{equation}
where $ v(\cdot) \in \bar{V} $ .\\
It is readily demonstrable that strategy
 (\ref{DAV7}) fulfills the relations
\begin{equation}\label{DAV8}
c(t)U(t, v)= c(t)v(t) - \phi(t)\varepsilon ,
\end{equation}
\begin{equation}\label{DAV9}
||U(t,v)||^2 = ||v(t)||^2 +  \lambda e^{2k \int_{0}^{t} c(s)ds} ,  for \ \ \ all\ \ \ t \geq 0, 
\end{equation}
whereas $$ \phi(t) = c(t)\left(\langle v(t), \varepsilon \rangle \pm  \sqrt{\lambda e^{2k \int_{0}^{t} v(s)ds}+\langle v(t), \varepsilon\rangle^{2}}\right).$$
With the help of (\ref{DAV8}), (\ref{DAV9}), and Lemma (\ref{TK}), it is possible to determine if the strategy (\ref{DAV7}) is permissible.
 Indeed  
 \begin{align*}
|U(t, v)|^2 &= |v(t)|^2 + \lambda e^{2k \int_{0}^{t} c(s)ds}
\\&\leq \sigma^2 + \lambda e^{2k \int_{0}^{t} c(s)ds} + 2k\int_{0}^{t} c(s)|v(t)|^2 ds
\\& =  \sigma^2 + \lambda + 2k\int_{0}^{t} c(s)\lambda e^{2k \int_{0}^{s} c(r)dr}ds + 2k\int_{0}^{t} c(s)|v(t)|^2 ds
\\& =  \rho^2 + 2k\int_{0}^{t} c(s)\left(|v(s)|^2 + \lambda e^{2k \int_{0}^{s} c(r)dr}\right)ds
\\& =  \rho^2 + 2k\int_{0}^{t} c(s)|U(s, v(s)|^2ds. 
\end{align*}
Explicitly, $ \displaystyle |U(t, v)|^2 \leq \rho^2 + 2k\int_{0}^{t} c(s)|U(s, v(s)|^2ds.$\\
Assuming the validity of the theorem's hypothesis, strategy (\ref{DAV7}) holds for any $ t $ between the interval $ (0,\tau ]$  for $\tau $ in $ (0, \theta]$.\\
Suppose,
   \begin{equation}\label{DAV12}
 U(t,v) = v(t),
\end{equation}
where the time $x(\tau) = y(\tau)$ is denoted by $\tau$, strategy (\ref{DAV12}) is permissible in the time interval $ [\tau, \theta]$ given that
 \begin{align*}
|U(t, v)|^2 &= \sigma^2 + 2k\int_{\tau}^{t} c(s)|v(s)|^2ds
\\&\leq \rho^2 + 2k\int_{0}^{t} c(s)\left(|U(s , v)|^2 - \lambda e^{2k \int_{0}^{s} c(r)dr} \right)ds
\\&\leq \rho^2 + 2k\int_{0}^{t} c(s)|U(s , v)|^2 ds.
\end{align*}
Now let $ x_{0} \neq y_{0} $. From (\ref{DAV3}) and (\ref{DAV4}) we have $ y(t)- x(t)= \varepsilon h(t)$ whereas\\
$\displaystyle h(t) = |y_{0} -x_{0}|e^{-\alpha t } + \int_{0}^{t} e^{-\alpha(t - s)}c(s)\langle v(s),\varepsilon\rangle ds -\int_{0}^{t} e^{-\alpha(t - s)}c(s)\sqrt{\lambda e^{2k \int_{0}^{t} c(s)ds} + \langle v(s), \varepsilon \rangle^2 }ds. $\\
Clearly, $ h(0) = |y_{0} -x_{0}|> 0 $, and hence the condition of Theorem (\ref{TTT}) follows if we can establish $ h(\vartheta)\leq 0,$ this will means that, for some $\tau\in [0,\vartheta ]$, $h(\tau)= 0 $,
 \\
to that aim, we take into account the subsequent vector function:
$$ g(t,s): = \left(\sqrt{\lambda} e^{-\alpha(t - s)}c(t)e^{k \int_{0}^{t} c(s)ds} ,   e^{-\alpha(t - s)}c(t)\langle v(t), \varepsilon \rangle \right).$$\\
Therefore \\

\begin{align*}
  \int_{0}^{\vartheta} e^{-\alpha(\vartheta - s)}c(s)\sqrt{\lambda e^{2k \int_{0}^{t} c(s)ds} + \langle v(s), \varepsilon \rangle^2}ds &= \int_{0}^{\vartheta}|g(s)|ds 
\\&\geq \left\vert\int_{0}^{\vartheta}g(s)ds\right\vert
\\& = M^{\frac{1}{2}}    ,
\end{align*}

where $ \displaystyle M   = \lambda\left(\int_{0}^{\vartheta} e^{-\alpha(\theta - s)}c(s)e^{k \int_{0}^{s} c(r)dr}ds \right)^2 +\left(\int_{0}^{\vartheta} e^{-\alpha(\vartheta - s)}c(s)\langle v(s), \varepsilon \rangle ds \right)^2 $ .\\
Then \\
\begin{equation}\label{DAV13}
 h(\vartheta) \leq |y_{0} -x_{0}|e^{-\alpha \vartheta } + \int_{0}^{\vartheta} e^{-\alpha(\vartheta - s)}c(s)\langle v(s), \varepsilon \rangle ds - M^{\frac{1}{2}}.
\end{equation}
The condition $ y(\vartheta)\in X $ implies that
\begin{equation}\label{DAV14}
 2\langle y_{0}e^{-\alpha\vartheta} - x_{0}e^{-\alpha\vartheta} , y(\vartheta) \rangle \leq \Delta + |y_{0}e^{-\alpha\vartheta}|^2  -|x_{0}e^{-\alpha\vartheta}|^2.
\end{equation}
Since $ y_{0}e^{-\alpha\vartheta} - x_{0}e^{-\alpha\vartheta} = |y_{0}e^{-\alpha\vartheta} - x_{0}e^{-\alpha\vartheta}| \varepsilon $ and using (\ref{DAV3}), we have\\
 \begin{align*}
 \langle  \varepsilon  , y(\vartheta) \rangle &= \frac{\langle y_{0}e^{-\alpha\vartheta} - x_{0}e^{-\alpha\vartheta},y(\vartheta)  \rangle}{|y_{0}e^{-\alpha\vartheta} - x_{0}e^{-\alpha\vartheta}|}  
\\& \leq \dfrac{ \Delta + |y_{0}e^{-\alpha\vartheta}|^2  -|x_{0}e^{-\alpha\vartheta}|^2}{2|y_{0}e^{-\alpha\vartheta} - x_{0}e^{-\alpha\vartheta}|} = d.
\end{align*} 
Also,
\begin{align*}
  \langle  \varepsilon  , y(\vartheta) \rangle &= \langle  \varepsilon , y_{0}e^{-\alpha\vartheta} + \int_{0}^{\vartheta} e^{-\alpha(\vartheta - s)}c(s)v(s)ds \rangle 
\\& =\langle \varepsilon , y_{0}e^{-\alpha\vartheta} \rangle + \int_{0}^{\vartheta} e^{-\alpha(\vartheta - s)}c(s)\langle v(s) ,\varepsilon \rangle ds.
\end{align*}
 This implies
\begin{equation}\label{DAV16}
\int_{0}^{\vartheta} e^{-\alpha(\vartheta - s)}c(s)\langle v(s) ,\varepsilon \rangle ds \leq d - \langle  \varepsilon  , y_{0}e^{-\alpha\vartheta} \rangle.
\end{equation}
Since the function $\psi(t) = |y_{0}e^{-\alpha\vartheta} - x_{0}e^{-\alpha\vartheta}| + t -(\Delta + t^2)^\frac{1}{2} $ is a function of $ t $ that increases, therefore (\ref{DAV13}) and (\ref{DAV16}) imply that,
\begin{equation}\label{DAV17}
h(\vartheta) \leq |y_{0}e^{-\alpha\vartheta} - x_{0}e^{-\alpha\vartheta}|  + d -\langle \varepsilon ,y_{0}e^{-\alpha\vartheta} \rangle - \left(\Delta + (d- \langle  \varepsilon  , y_{0}e^{-\alpha\vartheta} \rangle )^2\right)^\frac{1}{2}.
\end{equation}
It can be verified that\\
$$ (|y_{0}e^{-\alpha\vartheta} - x_{0}e^{-\alpha\vartheta}| + (d -\langle  \varepsilon  ,y_{0}e^{-\alpha\vartheta} \rangle)^2 = \Delta + (d- \langle  \varepsilon  , y_{0}e^{-\alpha\vartheta} \rangle )^2. $$
Hence $ h(\vartheta) \leq 0 $. Thus, $ h(\tau) = 0$ for certain
 $ \tau $. In other words, $x(\tau) = y(\tau)$.
\end{proof}

\noindent \textbf{Evasion Differential Game}

\begin{theorem}\label{TTTk}
In game (\ref{AT3})- (\ref{DAV2}), evasion is possible whenever $ \rho < \sigma $.
\end{theorem}
\begin{proof}
 	Allow the evader applies  strategy
 	 \begin{equation}\label{DAV20}
 	 V(t) = -\sigma e^{k\int_{0}^{t} c(s)ds} \varepsilon, 
 	\end{equation}
 	where $\varepsilon = \frac{y_{0} - x_{0}}{|y_{0} - x_{0}|}$.\\
 The permissible of strategy (\ref{DAV20}) is guaranteed, in view of the fact that, from lemma (\ref{TK}), it implies that
 	\begin{align*}
 	 \vert V(t)\vert^{2} &=\sigma^{2} e^{2k\int_{0}^{t} c(s)ds} \vert \varepsilon \vert^{2}
 	\\& = \sigma^2\left(1 + 2k\int_{0}^{t} c(s) e^{2k\int_{0}^{s} c(r)dr}ds \right)
 	\\& = \sigma^2 + 2k\int_{0}^{t} c(s) \sigma^2 e^{2k\int_{0}^{s} c(r)dr}ds 
 		\\& = \sigma^2 + 2k\int_{0}^{t} c(s) | V(s)|^2ds.
 	\end{align*} 
 Suppose the evader uses strategy (\ref{DAV20}), we prove that $\vert x(t)-y(t)\vert > 0 $ for every $t$.
 Certainly,

 	\begin{align*}                          
	| x(t)-y(t)| &= \left\vert  x_{0}e^{-\alpha t} - y_{0}e^{-\alpha t} -\int_{0}^{t} e^{-\alpha(t - s)} c(s)V(s)ds + \int_{0}^{t} e^{-\alpha(t - s)} c(s)u(s)ds\right\vert 
 	\\& \geq  \left\vert  x_{0}e^{-\alpha t} - y_{0}e^{-\alpha t} -\int_{0}^{t} e^{-\alpha(t - s)} c(s)V(s)ds \right\vert - \left\vert \int_{0}^{t} e^{-\alpha(t - s)} c(s)u(s)ds\right\vert
 	\\& \geq \left\vert  x_{0}e^{-\alpha t} - y_{0}e^{-\alpha t} -\int_{0}^{t} e^{-\alpha(t - s)} c(s)V(s)ds \right\vert - \int_{0}^{t} e^{-\alpha(t - s)} c(s)|u(s)|ds 
 	\\& \geq |x_{0}e^{-\alpha t} - y_{0}e^{-\alpha t}| + \sigma\int_{0}^{t} e^{-\alpha(t - s)} c(s) e^{k\int_{0}^{s} c(r)dr}ds - \rho\int_{0}^{t} e^{-\alpha(t - s)} c(s)e^{k\int_{0}^{s} c(r)dr}ds
 	\\& = |x_{0}e^{-\alpha t} - y_{0}e^{-\alpha t}| +(\sigma - \rho) \int_{0}^{t} e^{-\alpha(t - s)} c(s) e^{k\int_{0}^{s} c(r)dr}ds > 0 .
 	\end{align*} 
Therefore, $x(t) \neq y(t)$ for all $t > 0$, which concludes the proof of Theorem \ref{TTTk}.
 \end{proof}

\section{conclusion}
\noindent In this work, we analyzed a differential game of fixed duration in the Hilbert space $\mathbb{R}^n$. The dynamics of the players are governed by a generalized first-order linear differential equation, and their control functions are constrained by Gronwall-type conditions.\\[3mm]
\noindent In Theorem \ref{TTT}, we established that if $y(\vartheta) \in X$ and $\rho > \sigma$, then the pursuit is completed. In other words, when the evader’s state at time $\vartheta$ lies in the half-space $X$ and the pursuer’s maximum speed exceeds that of the evader, it follows that $x(\vartheta) = y(\vartheta)$.\\[3mm]
\noindent Furthermore, in Theorem \ref{TTTk}, we proved that evasion is possible in the game (\ref{AT3})-(\ref{DAV2}) whenever $\rho \leq \sigma$, meaning that the pursuer’s maximum speed is less than or equal to that of the evader.\\
This work extends and generalizes earlier result by Rilwan et al. \cite{TY21}. The theoretical framework developed here has potential applications in engineering systems, economic competition models, missile guidance, autonomous vehicles, robotics, and related multi-agent control problems.

\section{Declarations}
\textbf{Competing Interest}: The authors declare no conflicts of interest.\\
\textbf{Author Contributions}: David Terna Gbande and Jamilu Adamu drafted the paper; Abbas Ja'afaru Badakaya and Mehdi Salimi supervised and reviewed it.\\
\textbf{Funding}: This research received no funding.\\
\textbf{Availability of data and material}: No data is associated with this work.\\

\end{document}